\DeclareSymbolFont{AMSb}{U}{msb}{m}{n} 			
\documentclass[12pt,oneside,noamsfonts,a4paper]{amsart}

\usepackage[utf8]{inputenx}  
\usepackage[T1]{fontenc}     

\usepackage[USenglish]{babel}

\usepackage[charter,expert,sfscaled]{mathdesign}  
\renewcommand{\in}{\smallin}
\renewcommand{\notin}{\notsmallin}
\renewcommand{\setminus}{\smallsetminus}

\usepackage[scaled=.96,sups]{XCharter}            
\usepackage[kerning]{microtype}
\DeclareMicrotypeAlias{XCharter-TLF}{bch} 

\usepackage{mathtools}       
\usepackage{url}             

\usepackage{mathscinet}			

{\newlength\figurewidth}
\newcommand\fdash{\settowidth\figurewidth{9}\raisebox{0.6ex}{\makebox[\figurewidth]{\hrulefill}}}

\newtheorem{theorem}{Theorem}
\newtheorem{proposition}[theorem]{Proposition}
\newtheorem{lemma}[theorem]{Lemma}
\newtheorem{announcement}[theorem]{Announcement}

\newtheorem*{claim*}{Claim}

\newtheorem{corollary}[theorem]{Corollary}

\theoremstyle{definition}
\newtheorem{definition}[theorem]{Definition}

\newtheorem*{problem}{Problem}

\newcommand{\pw}[1]{\mathcal{P}\left(#1\right)}  
\newcommand{\iter}{\mathbin{*}}          
\newcommand{\os}{\mleft\{ \,}             
\newcommand{\cs}{\, \mright\}}             
\newcommand{\la}{\left\langle \,}           
\newcommand{\ra}{\, \right\rangle}          
\usepackage{mleftright}                       
\newcommand{\card}[1]{\mleft| #1 \mright|}    

\DeclareMathOperator{\cov}{cov}
\DeclareMathOperator{\cof}{cof}

\DeclareMathOperator{\dom}{dom}
\DeclareMathOperator{\cod}{cod}
\newcommand{\join}{\mathbin{\mathrm{join}}}

\renewcommand{\restriction}{\mathbin{\!\upharpoonright}}   

\renewcommand{\mid}{\shortmid}             

\title[No P-points in Silver extensions]{There are no P-points in Silver extensions}

\author{David Chodounský}
\address{Institute of Mathematics of the Czech Academy of Sciences,
Žitná~25, \linebreak[2] Praha~1, Czech Republic}
\email{chodounsky@math.cas.cz}

\author{Osvaldo Guzmán} 
\address{University of Toronto,
27 King's College Circle, Toronto, Canada}
\email{oguzman@math.utoronto.ca}

\subjclass[2010]{Primary: 03E35}
\keywords{P-point, Silver forcing, Silver model, canonical model}

\thanks{
	The first author was supported by the GACR
	project 17\fdash33849L and RVO:\ 67985840.
	The second author was supported by NSERC grant number~455916 
	and his visit to Prague was funded by the GACR
	project 15\fdash34700L and RVO:\ 67985840.
	Some of the work and consultations with J.\ Zapletal were conducted during the
	ESI workshop \emph{Current Trends in Descriptive Set Theory} held in December 2016 in Vienna.
}

\begin{document}

\begin{abstract}
	We prove that after adding a Silver real no ultrafilter
	from the ground model can be extended to a P-point,
	and this remains to be the case in any further extension
	which has the Sacks property.
	We conclude that there are no P-points in the Silver model.
	In particular, it is possible to construct a model without P-points
	by iterating Borel partial orders.
	This answers a question of Michael Hrušák.
	We also show that the same argument can be used for
	the side-by-side product of Silver forcing. This provides
	a model without P-points with the continuum arbitrary large,
	answering a question of Wolfgang Wohofsky.
\end{abstract}

\maketitle

\noindent
The first author dedicates this work to his teacher, mentor and dear friend Bohuslav Balcar.
The crucial result was proved on the day of his passing.

\section*{Introduction}\label{sec:intro}

\noindent
Ultrafilters on countable sets have become of great importance in infinite combinatorics.
A non-principal ultrafilter $\mathcal{U}$ is called a \emph{P-point} if
every countable subset of $\mathcal{U}$ has a pseudointersection in $\mathcal{U}$. 
Recall that a set $X\subseteq\omega$ is called a \emph{pseudointersection} of a
family $\mathcal{B} \subseteq {[\omega]}^{\omega}$ if $X \setminus B$ 
is finite for every $B\in\mathcal{B}$.
Ultrafilters of this special type have been extensively studied in set theory and topology. 
Walter Rudin in 1956 (see~\cite{Rudin}) proved that the topological space 
$\omega^* = \beta\omega \setminus \omega$ is not homogeneous assuming the continuum
hypothesis $\mathsf{CH}$. It is well known that the non-principal ultrafilters correspond in
a natural way to points of $\omega^*$ and P-points are exactly points 
with neighborhoods closed under countable intersections. 
Rudin proved the non-homogeneity 
of $\omega^*$ using the following argument: $\mathsf{CH}$ implies
that P-points exist, ultrafilters that are not P-points always exist,
and a P-point and a non-P-point have different topological types. 
Frolík established in~1967 (see~\cite{Frolik}) that
$\omega^*$ is not homogeneous without the need of the continuum hypothesis. 
Although Frolík's proof does not provide any specific types of ultrafilters, 
various distinct topological types of ultrafilters were later identified even without 
using any additional set-theoretic assumptions, see~\cite{Kunen-weak,vMill-16,Verner-lonely}. 

Nevertheless, P-points remain one of the central objects of research of 
set theorists and topologists. P-points are fundamental in forcing
theory; most of the methods of preserving an ultrafilter in generic 
extensions require preserving some kind of a P-point, the reader may consult 
e.g.~\cite{Barty, Zapletal-preserving} for more details. 
They also appear in the study of the
Tukey order~\cite{NatashaSurvey},
partition calculus~\cite{BaumgartnerTaylor}, 
model theory~\cite{Blass-RK},
and other topics.
The study of P-points is still a rich and active area of study, 
the reader may consult e.g.~\cite{Booth,MichaelVerner,EmbeddingPPoints}
for some more results and applications of P-points. 

\medskip

A well known result of Ketonen states that it is possible to construct a P-point if 
the dominating number is equal to the size of the continuum; 
$\mathfrak{d} = \mathfrak{c}$\footnote{
The dominating number $\mathfrak d$ is the least cardinality of a set of functions 
in $\omega^\omega$ such that every function is eventually dominated by 
a member of that set. $\mathfrak c$ is the cardinality of the continuum.
}~\cite{Ketonen}. 
It is also possible to construct P-points if the parametrized diamond principle
$\Diamond \mleft( \mathfrak{r} \mright)$\footnote{
The \emph{reaping number} $\mathfrak{r}$ is the smallest size of a
family $\mathcal{R}\subseteq {[\omega]}^{\omega}$ such that for
every $X\in {[\omega]}^{\omega}$ there is $R\in\mathcal{R}$ such
that either $R \subseteq X$ or $R \subseteq\omega\setminus X$. For more
information on the reaping number and cardinal characteristics of the continuum in general,
the reader may consult~\cite{BlassHandbook}.
The formulation of the associated diamond principle
$\Diamond \mleft( \mathfrak{r} \mright)$ is somewhat involved  
and since it is not used in the present paper, it is omitted.
}
holds, see~\cite{ParametrizedDiamonds} for more information on parametrized diamond principles.
On the other hand a remarkable theorem of
Shelah states that the existence of P-points cannot be proved
using just the axioms of $\mathsf{ZFC}$ alone.
This result was proved in 1977 and first published in~\cite{Wimmers}. 
The reader may find the proof in~\cite{Shelah-book}.
The model of Shelah is obtained by iterating the Grigorieff forcing
with parameters ranging over non-meager P-filters.

\smallskip

Independence results are often demonstrated in models obtained 
by employing forcing iterations of definable posets.
One possible formalization of such canonical models is treated 
in~\cite{ParametrizedDiamonds}.
We say that a partial order $(P, \leq)$ is Borel if there is a
Polish space $X$ such that $P$ is a Borel subset of $X$, and $\leq$ is a
Borel subset of $X \times X$.
A \emph{canonical model} is a model obtained by performing 
a countable support forcing iteration of Borel proper partial orders
of length~$\omega_2$.
At the \emph{Forcing and its applications retrospective workshop}
held at the Fields Institute in~2015 Michael Hrušák posed the following problem.

\begin{problem}
	Do P-points exist in every canonical model?
\end{problem}

A canonical model will contain a P-point if the steps of the iteration add unbounded reals 
or if no splitting reals are added---in the resulting model either
$\mathfrak{d} = \mathfrak{c}$ or $\Diamond (\mathfrak{r})$ does hold.
Consequently, one only needs to consider Borel
$\omega^{\omega}$-bounding forcing notions
which do add splitting reals.
The best known examples of this type of forcing are the random poset and the Silver poset.
We answer the question of Hrušák
in negative; Theorem~\ref{thm:Silver-iteration}
states that there are no P-points in the Silver model.

In~\cite{Cohen-random} it was claimed that there is a P-point
in the random model.
Unfortunately, the presented proof is incorrect
and the existence of P-points in the random model is presumably unknown.
We will address this issue in the Appendix section.

\begin{problem}
	Are there P-points in the random model?
\end{problem}

The existence of a model without P-points with the continuum
larger that $\omega_2$ was an open question~\cite{Wohofsky-thesis}.
Theorem~\ref{thm:Silver-product} states that forcing
with the side-by-side product of Silver forcing produces such a model.

\medskip

Our notation and terminology is mostly standard, 
including some folklore abuse of notation.
When $p$ is a partial function from $\omega$ to $2$, 
we denote this by $p; \omega \to 2$
and we write
$p^{-1}(1)$ instead of $p^{-1}[\os1\cs]$. 
We say that $\mathcal{I} = \os  I_{n} \mid n\in\omega \cs$ 
is an \emph{interval partition} if there is an increasing sequence of natural
numbers ${\langle m_{n} \rangle}_{n \in \omega}$ such that $m_{0} = 0$
and $I_{n}=[m_{n},m_{n+1})$. 

We say that a forcing notion $\mathbf{P}$ has the \emph{Sacks property} if for
every $p\in\mathbf{P}$ and for every $f$ such that
$p \Vdash \dot{f} \in \omega^{\omega}$ there is $q \leq p$ and 
$\os X_{n} \mid n \in \omega \cs$ such that 
$X_{n}\in{[\omega]}^{n+1}$ for every $n\in\omega$, and 
$q \Vdash \dot{f}(n)  \in X_{n}$ for each $n \in \omega$. 
It is a common practise to require in the 
definition of the Sacks property that $X_{n} \in {[\omega]}^{2^{n}}$, 
instead of $X_{n} \in {[\omega]}^{n+1}$ as we demanded.
Nevertheless, both resulting notions are equivalent, 
see e.g.~\cite[section~3]{SacksForcingandtheSacksProperty}.

If $p; \omega \to 2$ is a partial function we denote by by $[p]$ 
the set of all total function extending $p$, i.e.\ 
$[p]  = \os f\in 2^{\omega} \mid p \subseteq f \cs$.

\section*{Destroying P-points with Silver reals}\label{sec:kill}

\noindent
For a partial function $p; \omega \to 2$ we denote $\dom p$ the domain of $p$ and 
$\cod p = \omega \setminus \dom p$.
We denote the \emph{Silver forcing} (after Jack Howard Silver, see~\cite{Mathias-survey}) 
by $\mathbf{PS}$. Some authors also call this poset the Prikry--Silver forcing.
It consists of all
partial functions $p; \omega \to 2$ such that $\cod p$ is infinite, 
and relation $p \leq q$ is defined as $q \subseteq p$.
We will always assume that $p^{-1}(1)$ is infinite for each $p \in \mathbf{PS}$,
such conditions form a dense subset of the poset. 
If $G$ is a generic filter for the Silver forcing, 
the \emph{Silver generic real} is defined as 
$r = \bigcap \os [p]  \mid p\in G\cs$. 
It is well known and easy
to see that $G$ and $r$ can be defined from each other. 
A typical application of the Silver forcing is to demonstrate that the inequality
$\cof \mathcal{N} < \mathfrak{r}$ is consistent.\footnote{
$\mathcal{N}$ is the ideal of Lebesgue null subsets of the real line.
}
The reader may consult~\cite{Halbeisen} for an introduction and
more information regarding the Silver forcing. 
It is well known that the Silver forcing is proper and has the Sacks property.

For a partial (or total) function $p; \omega \to 2$
define an interval partition of~$\omega$ by letting
$I_n(p) = \os k \in \omega \mid \card{k \cap p^{-1}(1)} = n \cs$ for $n \in \omega$ 
and $\mathcal{I} (p)  =  \os I_{n} (p)  \mid n \in \omega\cs$. 
Note that if $q$ extends $p$, then $\mathcal{I}( q)$ refines $\mathcal{I}(p)$, 
i.e.\ every interval of $\mathcal{I}(p)$ is the union of intervals of $\mathcal{I}(q)$. 
Moreover, if $r$ is the generic real,
then $\mathcal{I}(r)$ refines $\mathcal{I}(p)$ for every $p$ in the
generic filter. The proofs of the following simple observations are left for the reader.

\begin{lemma}\label{lem:help1}
Let $p \in \mathbf{PS}$ and $k \in\omega$ be such that 
$I_{k}(p) \subseteq \dom p$. 
\begin{enumerate}
\item If $q \leq p$, then $I_{k}(p) \in \mathcal{I}(q)$.
\item $p\Vdash I_{k}(p) \in \mathcal{I}(\dot{r})$ 
(where $\dot{r}$ is the name for the generic real).
\end{enumerate}
\end{lemma}

\begin{lemma}\label{lem:help2}
Assume that $p,q \in \mathbf{PS}$ and $k,n\in\omega$ are such that
\begin{enumerate}
\item $q\leq p$,
\item 
$I_k(p) \in \mathcal{I}(q)$, and
\item $\card{ q^{-1}(1) \cap \min I_{k}(p)} = 
\card{ p^{-1}(1) \cap \min I_{k}(p)} + n$.
\end{enumerate}
Then $I_{k}(p) = I_{k+n}(q)$.
\end{lemma}

As a consequence of these observations we conclude the following.

\begin{corollary}\label{cor:help}
If $p\in \mathbf{PS}$ and $k\in\omega$ are such that 
$I_{k}(p) \subseteq \dom p$, then $p$ forces that:
There is $n \in \omega$, $n \leq\card{ \cod p  \cap \min I_{k}(p)}$ 
such that $I_{k}(p) = I_{k+n}(\dot{r})$ 
(where $\dot{r}$ is the name for the generic real).
\end{corollary}

By $-_{n}$ and $=_{n}$ we denote
the subtraction operation and congruence relation modulo~$n$.
The notation $k \in_n X$ is interpreted as `there is $x \in X$ such that $k =_n x$.'
For $X, Y \subset n$ we write $X -_n Y = \os x -_n y \mid x \in X, y \in Y \cs$.

\begin{lemma}\label{lem:shift}
    For each $n \in \omega$ there exists $k(n) \in \omega$ such that
    for each set $C \in {[k(n)]}^n$ there exists $s \in k(n)$
    such that $C \cap \mleft(C -_{k(n)} \os s \cs\mright) = \emptyset$.
\end{lemma}
\begin{proof}
	If $s$ does not satisfy the conclusion of the lemma,
	then $s \in C -_{k(n)} C$. As $\card{C -_{k(n)} C} \leq n^2$,
	any choice of $k(n) > n^2$ works as desired.
\end{proof}

The following proposition contains the main technical argument 
central for the results of this paper.

\begin{proposition}\label{prop:kill}
    Let $\mathcal U$ be a non-principal ultrafilter and $\dot{Q}$
    be a $\mathbf{PS}$-name for a forcing
    such that $\mathbf{PS} \iter \dot{Q}$ has the Sacks property.
    If $G\subset \mathbf{PS} \iter \dot{Q}$ is a generic filter over $V$,
    then $\mathcal{U}$ cannot be extended to a P-point in $V[G]$.
\end{proposition}

Before giving the formal the proof of the Proposition, 
let us sketch the core idea of the argument. 
The basic approach is the same as in the no-P-points proof of Shelah from~\cite{Shelah-book}. 
We will show that in order to extend $\mathcal{U}$ to an ultrafilter in the generic extension, 
one would need to add to $\mathcal U$ a particular countable set $\mathcal D$ 
of newly introduced subsets of $\omega$, and at the same time there is no way to 
add to $\mathcal U$ also the pseudointersection of $\mathcal D$; 
for each pseudointersection $Z$ of $\mathcal D$ there is $U \in \mathcal U$ 
such that $Z \cap U = \emptyset$.

The sets in $\mathcal D$ will be chosen as the typical independent reals 
added by the Silver forcing. Let $r$ be the generic real, 
define $d^n_i$ as the union of intervals ${I}_j(r)$ such that
$j =_{n} i$. 
Although it is easy to see that each such $d^n_i$ is an independent real, 
this fact will not be explicitly needed in our argument  
and is therefore left for the interested reader to observe. 
For a fixed $n$ the sets $d^n_i$ form a partition of $\omega$ into $n$ pieces, 
and consequently each ultrafilter extending $\mathcal U$ needs to contain 
one element of this partition, denote this set $d^n_{y(n)}$. 
We will show that the set $\mathcal D = \os d^n_{y(n)} \mid n \in \omega \cs$ 
works as desired.

The argument for non-existence of pseudointersections will go 
along the lines of the following simple claim.

\begin{claim*}
Suppose $\mathcal D = \os d^n \mid n \in \omega\cs$ is a subset 
of an ultrafilter $\mathcal U$ with the following property. 
For every function $f\colon \omega \to \omega$ there is an interval partition 
$\os a_n \mid n \in \omega \cs$ such that 
\begin{itemize}
 \item $f(n) < \min a_{n+1}$ for each $n \in \omega$, and
 \item $\bigcup \os a_n \cap d^n \mid n \in \omega \cs \notin \mathcal U$.
\end{itemize}
Then $\mathcal D$ does not have a pseudointersection in $\mathcal U$, 
and consequently $\mathcal U$ is not a P-point.
\end{claim*}

Although the details of the sketched idea will be for technical reasons somewhat 
adjusted, e.g.\ we will use only a subset of the above defined set $\mathcal D$,
the formal proof of Proposition~\ref{prop:kill} 
will roughly follow the described argument.

\begin{proof}
    First we use the function $k$ from Lemma~\ref{lem:shift}
    to inductively construct two increasing sequences of integers.
    Put $v(0) = 0$ and $m(0) = k(2)$.
    Assume $v(n-1)$, $m(n-1)$ are defined,
	put $v(n) = \sum \os m(i) \mid i \in n \cs $
    and $m(n) = k\big((n+1) (v(n) + 2)\big)$.
    Let $r$ be the $\mathbf{PS}$ generic real in $V[G]$
    added by the first stage of the iteration.
    For $n \in \omega$ and $i \in m(n)$ let
    	\[ D^n_i = \bigcup \os I_j(r) \mid j \in \omega, j =_{m(n) } i \cs.\]
    For a fixed $n$ the set $\os D_{i}^{n}\mid i< m (n)  \cs$ 
    is a finite partition of $\omega$.
    We will show that in $V[G]$, for every function
    $y \colon \omega \to \omega$ which satisfies $y (n) <m(n)$ for every $n\in\omega$, 
    and every pseudointersection $Z$ of
    $\os D^n_{y(n)} \mid n \in \omega \cs$ there is a set $U \in \mathcal U$
    such that $U \cap Z = \emptyset$.
    This implies that $\mathcal U$ cannot be extended to a P-point in $V[G]$.

    Let $(p, \dot q)$ be any condition in $\mathbf{PS} \iter \dot{Q}$,
    and let $\dot Z$, $\dot y$ be the corresponding names for $Z$ and $y$.
    Utilizing the Sacks property we can assume that
    there are $f \colon \omega \to \omega$ and
    $\os X_n \in {[m(n)]}^{n+1} \mid n \in \omega \cs$ in $V$
    such that
    	\[(p, \dot q) \Vdash \mleft(\dot Z \setminus f(n)\mright) \subseteq D^n_{\dot{y}(n)}
    	\text{ \ and \ } \dot{y}(n) \in X_n.\]

    Choose an interval partition
    $\mathcal A = \os A_n \mid n \in \os -1 \cs \cup \omega \cs$
    of $\omega$ ordered in the natural way such that
    \begin{enumerate}
        \item $f(n) < \min A_{2n}$ for each $n \in \omega$,
        \item $m(n) < \card{A_{2n+j} \cap \cod p}$ for each $n \in \omega$, $j \in 2$, and
        \item $\mathcal I(p)$ refines $\mathcal A$.
    \end{enumerate}
    We will assume that $U_0 = \bigcup \os A_{2n+1} \mid n \in \omega\cs \in \mathcal U$,
    otherwise take the interval partition
    $\mathcal A' = \la A_{-1} \cup A_0, A_1, A_2, \ldots \ra$ instead.\footnote{
    In the following proof, we will use the second assumption on the interval partition $\mathcal A$
    only for $j = 0$. Notice however, that assuming it only for $j=0$ at the moment of 
    choosing $\mathcal A$ would not have been sufficient as if it were the case that 
    $U_0 \notin \mathcal U$, we would be working with the partition $\mathcal A'$ instead, 
    and $\mathcal A'$ would not be fulfilling the necessary requirement. 
    The observant reader may also notice that the last assumption on $\mathcal A$ 
    will in fact not be necessary in the proof.
    }
    The plan is to use the trace of extensions of $p$ on the interval $A_{2n}$
    to control the possible behavior of the set 
    $D^n_{ {y(n)}} \cap A_{2n+1}$ for all $n\in\omega$ simultaneously.

    Let $p_1 \in \mathbf{PS}$ be any extension of $p$ such that
    $A_{2n-1} \subseteq \dom p_1$ and
    $\card{\cod p_1 \cap A_{2n}} = m(n)$ for each $n \in \omega$.    
    Note that for any $j \in\omega$ if $I_{j}(p_1) \subseteq A_{2n-1}$,
    then $p_1 \Vdash I_{j}(p_1) \in \mathcal{I}(\dot{r})$. 
    Also note that $\card{ \cod p_1 \cap \min A_{2n} } = v(n)$ for each $n\in\omega$.
	Let
	    \[
	    C_n =  X_n -_{m(n)} \os i \mid i \in v(n) + 2 \cs
	    \]
    and notice that $\card{C_n} \leq (n+1) (v(n) + 2)$.
    For $n\in \omega$ put
	    \[H_n = A_{2n+1} \cap
	    \bigcup \os I_j(p_1) \mid j \in \omega, j \in_{m(n)} C_n \cs .\]

    We will now distinguish two cases. Case~1;
    $\bigcup \os H_n \mid n \in \omega \cs \notin \mathcal U$,
    hence $U = \bigcup \os A_{2n+1} \setminus H_n \mid n \in \omega \cs \in \mathcal U$.
    Pick any $p_2 < p_1$, $p_2 \in \mathbf{PS}$ such that
    $p_2^{-1}(1) = p_1^{-1}(1)$ and
    $\card{\cod p_2 \cap A_{2n}} = 1$ for each $n \in \omega$. 
    Notice that $\mathcal{I}(p_1) = \mathcal{I}(p_2)$, 
    $\card{\cod p_2 \cap \min(A_{2n+1})} = n+1$,      
    and if $j\in\omega$ is such that 
    $I_{j}(p_{2}) \subseteq A_{2n+1}$, then
    $I_{j}(p_{2}) \subseteq \dom p_2$.
    For each $n \in \omega$ Corollary~\ref{cor:help} states that $p_2$ forces: 
    There is $i \leq n+1$ such that for each $j \in \omega$ 
    if $I_j(p_1) \subseteq A_{2n+1}$, 
    then $I_j(p_1) = I_{j + i}(\dot{r})$. 
    As $(p_2, \dot{q})$ forces $\dot{y}(n) \in X_n$,
    it follows that if $I_j(p_1) = I_{j + i}(r) \subseteq D^n_{ y(n) } \cap A_{2n+1}$,
    then $j \in_{m(n)}  \mleft( X_n -_{m(n)} \os i \mid i \in n+2 \cs \mright)
    \subseteq C_n$.
    We can conclude that:
		\[(p_2,\dot q) \Vdash D^n_{ \dot{y}(n) } \cap A_{2n+1} \subset H_n. \]
    This together with
    	\[ (p,\dot q) \Vdash (\dot Z \setminus \min A_{2n}) \subseteq D^n_{\dot{y}(n)}  \]
    implies that $(p_2,\dot q) \Vdash \dot{Z} \cap U = \emptyset$.

    Case~2; $U = \bigcup \os H_n \mid n \in \omega \cs \in \mathcal U$.
    Applying Lemma~\ref{lem:shift},
    for each $n \in \omega$ there exists $s_n \in m(n)$
    such that $C_n \cap \mleft(C_n -_{m(n)} \os s_n \cs \mright) = \emptyset$. 
    Put $t(n) =  \sum \os s_i \mid i \in n \cs \leq v(n) - n$. 
    Pick a condition $p_2 < p_1$, $p_2 \in \mathbf{PS}$ such that
    $\card{\cod p_2 \cap A_{2n}} = 1$ and
    	\[ \card{p_2^{-1}(1) \cap A_{2n}} = \card{p_1^{-1}(1) \cap A_{2n}} + s_n \]
    for each $n \in \omega$.
    Such $p_2$ exists as $\card{\cod p_1 \cap A_{2n}} = m(n)$. 
    Note that in this case 
    $\card{\cod p_2 \cap \min(A_{2n+1})} = n+1$,      
    and if $j\in\omega$ is such that 
    $I_{j}(p_{2}) \subseteq A_{2n+1}$, then
    $I_{j}(p_{2}) \subseteq \dom p_2$ 
    and $I_{j}(p_{2}) = I_{j - t(n+1)}(p_{1})$.
    For each $n \in \omega$ Corollary~\ref{cor:help} implies that $p_2$ forces:
    There is $i \leq n+1$ such that for each $j \in \omega$ 
    if $I_j(p_1) \subseteq A_{2n+1}$, 
    then $I_j(p_1) = I_{j + t(n+1) + i}(\dot{r})$. 
    As $(p_2, \dot{q})$ forces $\dot{y}(n) \in X_n$,
    it follows that if $I_j(p_1) = I_{j + t(n+1) +i }(r) \subseteq D^n_{ y(n) } \cap A_{2n+1}$,
    then \begin{multline*}
	  j \in_{m(n)}  
	  \mleft( X_n -_{m(n)} \os t(n+1) \cs \mright) -_{m(n)} \os i \mid i \in n+2 \cs = \\
	  = \mleft(\mleft( X_n -_{m(n)} \os t(n) \cs \mright) 
	  -_{m(n)} \os i \mid i \in n+2 \cs \mright) -_{m(n)} \os s_n \cs \subseteq \\
	  \subseteq \mleft( X_n -_{m(n)} \os i \mid i \in v(n) + 2 \cs\mright) -_{m(n)} \os s_n \cs
	  = C_n -_{m(n)} \os s_n \cs.
	  \end{multline*}    
    
    For $n \in \omega$ put
	    \[\bar{H}_n = A_{2n+1} \cap \bigcup \os I_j(p_1) \mid j \in \omega,
	    j \in_{m(n)} \mleft(C_n -_{m(n)} \os s_n \cs\mright) \cs, \]
    $H_n \cap \bar{H}_n = \emptyset$, because if $j \in_{m(n)} C_n $, 
    then $j \notin_{m(n)} C_n -_{m(n)} \os s_n \cs$.

    Now \[(p_2,\dot q) \Vdash D^n_{\dot{y}(n)} \cap A_{2n+1} \subset \bar{H}_n. \]
    Again, together with
    	\[ (p,\dot q) \Vdash (\dot Z \setminus \min A_{2n}) \subseteq D^n_{\dot{y}(n)} \]
    we get $(p_2,\dot q) \Vdash \dot{Z} \cap U = \emptyset$.
\end{proof}


The Silver model is the result of a countable support iteration
of Silver forcing of length~$\omega_2$.

\begin{theorem}\label{thm:Silver-iteration}
    There are no P-points in the Silver model.
\end{theorem}
\begin{proof}
    Denote by $\mathbf{PS}_{\alpha}$ the countable support iteration
    of Silver forcing of length $\alpha$ for $\alpha\leq\omega_{2}$.
    Assume $V$ is a model of $\mathsf{CH}$
    and let $G \subset \mathbf{PS}_{\omega_2}$ be a generic filter.
    Let $\mathcal{U} \in V[G]$ be a non-principal ultrafilter.
    For $\alpha < \omega_2$ let
    $\mathcal{U}_{\alpha} = \mathcal{U} \cap V[G_\alpha]$,
    where $G_\alpha$ is the restriction of $G$ to $\mathbf{PS}_\alpha$.
    By the standard reflection argument, there is $\alpha < \omega_2$
    such that $\mathcal{U}_\alpha \in V[G_\alpha]$
    and it is an ultrafilter in that model.
    Since the next step of the iteration adds a Silver real
    and the tail of the iteration has the Sacks property,
    Proposition~\ref{prop:kill} states
	that $\mathcal{U}_\alpha$ cannot be extended to a P-point
    in $V[G]$, in particular, $\mathcal{U}$ is not a P-point.
\end{proof}

We show that forcing with the side-by-side product of Silver forcing
also produces a model without P-points.

\begin{theorem}\label{thm:Silver-product}
	Assume $\mathsf{GCH}$,
	let $\kappa > \omega_1$ be a cardinal with uncountable cofinality.
	If $\bigotimes_\kappa \mathbf{PS}$ is the countable support product
	of $\kappa$ many Silver posets and
	$G \subset \bigotimes_\kappa \mathbf{PS}$ is a generic filter,
	then \[V[G] \models \text{ there are no P-points and } \mathfrak{c} = \kappa.\]
\end{theorem}
\begin{proof}
	It is well known that under $\mathsf{GCH}$ the poset
	$\bigotimes_\kappa \mathbf{PS}$ is an $\omega_2$-c.c.\ proper forcing notion,
	has the Sacks property (see e.g.~\cite{Koszmider-Sforcing}),
	and $V[G] \models \mathfrak{c} = \kappa$.
	Assume $\mathcal U$ is an ultrafilter in $V[G]$.
	Since $\bigotimes_\kappa \mathbf{PS}$ is $\omega_2$-c.c.,
	there is $J \subset \kappa$ of size $\omega_1$
	such for every $A \in \pw{\omega} \cap V$
	and $q \in \bigotimes_\kappa \mathbf{PS}$ 
	the statement $A \in \mathcal U$
	is decided by a condition
	with support contained in $J$
	and compatible with $q$.
	Choose $\alpha \in \kappa \setminus J$ and let $r$
	be the $\mathbf{PS}$ generic real added by the $\alpha$-th coordinate
	of the product.

	The theorem is now proved in the same way as Proposition~\ref{prop:kill};
	as the proof follows most parts of the proof 
	of Proposition~\ref{prop:kill} in verbatim, 
	we will focus in detail only on the points where adjustments are necessary. 
	
	Start with defining the functions $v$ and $m$, consider sets $D_i^n$ defined from
	the generic real $r$,
	and pick any $\bigotimes_\kappa \mathbf{PS}$ names $\dot{Z}, \dot{y}$. 
	Let $(p,q)$ be any condition in 
	$\mathbf{PS} \times { \bigotimes_{\kappa \setminus \os \alpha \cs}
	\mathbf{PS} } = \bigotimes_\kappa \mathbf{PS}$ 
	which forces that $\mathcal U$ is a non-principal ultrafilter; 
	we interpret $p \in \mathbf{PS}$ as the $\alpha$-th coordinate 
	and $q$ as the other coordinates of a condition in the full product poset.
	We invoke the Sacks property of $\bigotimes_\kappa \mathbf{PS}$ 
	to assume the existence of an appropriate function $f$ and 
	a sequence $\os X_n \mid n \in \omega \cs$. 
	Choose the interval partition $\mathcal A$ satisfying properties (1--3)
	with respect to $p$ and consider 
	$U_0 = \bigcup \os A_{2n+1} \mid n \in \omega\cs$.
	As $U_0 \in V$, there is a condition $(p,q_1) < (p, q)$ deciding 
	whether $U_0$ is an element of $\mathcal U$, 
	because of the choice of coordinate $\alpha$.
	We will assume that $(p,q_1) \Vdash U_0 \in \mathcal U$, 
	otherwise take the interval partition $\mathcal A'$ instead. 
	Follow with choosing the condition $p_1$ extending $p$, 
	define the sets $C_n$ and $H_n$ for $n \in \omega$. 
	
	Now consider the set $H = \bigcup \os H_n \mid n \in \omega \cs$. 
	As $H \in V$, there is $(p_1, q_2) < (p_1, q_1)$ deciding whether $H \in \mathcal U$. 
	Case~1; $(p_1, q_2) \Vdash H \notin \mathcal U$. 
	Now proceed again in verbatim as in case~1 
	of the proof of Proposition~\ref{prop:kill};
	define $U$, choose $p_2 < p_1$, 
	and conclude that $(p_2, q_2) \Vdash \dot{Z} \cap U = \emptyset$.
	
	Case~2; $(p_1, q_2) \Vdash H \in \mathcal U$.
	Proceed again as in case~2 
	of the proof of Proposition~\ref{prop:kill};
	define $U$, find $s_n$ for each $n \in \omega$, and choose $p_2 < p_1$.
	And finally conclude $(p_2, q_2) \Vdash \dot{Z} \cap U = \emptyset$.
\end{proof}

\section*{Concluding remarks}

\noindent
Theorem~\ref{thm:Silver-iteration}
can be stated in an axiomatic manner.
Recall that $\mathcal N$ denotes the ideal of Lebesgue null sets
and let $v_0$ be the ideal associated with the Silver forcing;
\[v_0 = \os A \subset 2^\omega \mid \forall p \in \mathbf{PS} \
\exists q \in \mathbf{PS}, q < p, [q] \cap A = \emptyset \cs.\]
This ideal was introduced in~\cite{CRSW} and studied in~\cite{Brendle-paradise, Prisco-Henle}. 
The proof of Proposition~\ref{prop:kill} can be reformulated to yield the following theorem, 
the detailed proof is provided in~\cite{Osvaldo-thesis}.
\begin{theorem}\label{thm:axiomatic-version}
The inequality $\cof \mathcal N < \cov v_0$ implies that there are no \mbox{P-points}.
\end{theorem}

\smallskip

An alternative version of results of this paper was suggested by Jonathan Verner.
The side-by-side product $\bigotimes_\omega \mathbf{PS}$ adds
a Silver generic real $r_\alpha$ for each coordinate $\alpha \in \omega$. 
Consider the pair of complementary splitting reals 
$X_\alpha^i = \bigcup \os I_{2n+i}(r_\alpha) \mid n \in \omega \cs$, $i \in 2$;
an argument similar to (and less technical than) the proof of Proposition~\ref{prop:kill} 
demonstrates the following.
\begin{claim*}
    Let $\mathcal U$ be a non-principal ultrafilter. 
    The product $\bigotimes_\omega \mathbf{PS}$ 
    forces that no pseudo-intersection of 
    $\os X_\alpha^{i(\alpha)} \mid \alpha \in \omega \cs$ is $\mathcal U$-positive, 
    and this remains to be the case in each further Sacks property extension.
\end{claim*}
Furthermore, it is possible to reason along the lines of the proof of 
Theorem~\ref{thm:Silver-product} to obtain a stronger version of the theorem. 
These results are to be included in forthcoming publications.

\begin{announcement}\label{thm:almost_Isbell}
    Assume $\mathsf{GCH}$,
    let $\kappa > \omega_1$ be a cardinal with uncountable cofinality.
    If $\bigotimes_\kappa \mathbf{PS}$ is the countable support product
    of $\kappa$ many Silver posets and
    $G \subset \bigotimes_\kappa \mathbf{PS}$ is a generic filter,
    then 
    \begin{multline*}
    V[G] \models \text{ For every non-principal ultrafilter } \mathcal U 
    \text{ there exists } \\ \os X_\alpha \mid \alpha \in \mathfrak c \cs \subset \mathcal U
    \text{ such that for each } y \in {[\mathfrak c]}^\omega \cap V \\
    \text{ no pseudointersection of } \os X_\alpha \mid \alpha \in y \cs 
    \text{ is an element of }\, \mathcal U.
    \end{multline*}
\end{announcement}

The motivation for stating this theorem 
comes from the problem of Isbell~\cite{Isbell} 
which asks for the existence of two Tukey non-equivalent ultrafilters on~$\omega$.
The problem can be equivalently formulated as a statement 
resembling the conclusion of Announcement~\ref{thm:almost_Isbell}, 
see~\cite{Natasha-Stevo}.

\begin{problem}[Isbell]
    Is it consistent that for each non-principal ultrafilter $\mathcal U$ on $\omega$
    there exists $\mathcal X \in {[\mathcal U]}^{\mathfrak c}$ 
    such that for each $\mathcal Y \in {[\mathcal X]}^\omega$ is 
    $\bigcap \mathcal Y \notin \mathcal U$?
\end{problem}

\section*{Appendix}

\noindent
At the request of the referee, we address here the situation 
concerning the random model. We point out the issue in the argument 
in~\cite{Cohen-random} used to reason for the existence of P-points 
in the random model. 
The reader may consult~\cite{GruffErrata,Fernandez-pathways} for more information.

It is an unpublished result of K.\ Kunen that if 
$\omega_1$ many Cohen reals are added to the ground model 
followed by adding $\omega_2$ many random reals, 
the resulting random model will contain a P-point. 
Recently A.\ Dow proved that P-points exist 
in the random model provided $\mathsf{CH}$ and $\square_{\omega_1}$ 
does hold in the ground model~\cite{Dow-random}.

The construction in~\cite{Cohen-random} uses the notion of a pathway. 
For a recent development and general treatment of pathways see~\cite{Fernandez-pathways}.

\begin{definition}
A sequence $\os A_{\alpha} \mid \alpha \in \kappa \cs$ is a 
\emph{pathway} if the following conditions hold.
\begin{enumerate}
\item $\omega^{\omega}= \bigcup\os A_\alpha \mid \alpha \in \kappa \cs$,
\item $A_{\alpha}\subseteq A_{\beta}$ for $\alpha < \beta$,
\item $A_{\alpha}$ does not dominate $A_{\alpha+1}$,\footnote{
I.e.\ there is a function in $A_{\alpha+1}$ 
not eventually dominated by any element of $A_\alpha$.
}
\item if $f,g\in A_{\alpha}$, then $(f \join g )\in A_{\alpha}$
(where $(f_0 \join f_1) \in \omega^\omega$  is defined by 
$(f_0 \join f_1)(2n+i) = f_i(n)$),
\item if $g$ is Turing reducible to $f$ and $f\in A_{\alpha}$, 
then $g\in A_{\alpha}$.
\end{enumerate}
\end{definition}

The following is~\cite[Theorem~1.1]{Cohen-random}.

\begin{theorem}
The existence of a pathway implies the existence of a $P$-point.
\end{theorem}

This result is a useful tool for proving the existence of P-points in
certain models. In order to prove that there is a P-point in the random
model (i.e.\ the model obtained by adding $\omega_{2}$ random reals to a model
of $\mathsf{CH}$), the author of~\cite{Cohen-random} aims to
construct a pathway in the generic extension.
We do not know whether there are pathways in this model.  
The construction from~\cite{Cohen-random} does not work, 
as we will demonstrate.

We denote $\mathbf{B}$ the random forcing 
and $\mathbf{B}(\omega_{2})$ the poset for adding $\omega_{2}$ 
many random reals. 
It is well known that if $M$ is a countable elementary submodel
of $H(\theta)$ (for some sufficiently large cardinal $\theta$), 
$r \colon \omega_{2}\to 2$ is a $\mathbf{B}(\omega_{2})$-generic 
function over $V$, and $\pi \colon \omega\to \omega_{2}$ 
is an injective function in $V$ (but not necessarily $M$), then 
$M[r\circ\pi]$ is a $\mathbf{B}$-generic extension of $M$ 
(see~\cite{Cohen-random} for more details).

We outline the construction in~\cite{Cohen-random}. 
Using $\mathsf{CH}$ in the ground model $V$ 
find  $\os M_{\alpha} \mid {\alpha\in\omega_{1}} \cs$, 
an increasing chain of countable elementary submodels of $H(\theta)$
such that 
$\omega^{\omega} \subset \bigcup \os M_{\alpha} \mid \alpha\in\omega_{1} \cs$. 
Let $r \colon \omega_{2}\to 2$ be a $\mathbf{B}(\omega_{2})$-generic function over $V$.
Work in $V[r]$; let $\Pi$ be the set of all injective functions from
$\omega$ to $\omega_{2}$ in $V$. 
For every $\alpha<\omega_{1}$ define 
$A_{\alpha}= \bigcup \os \omega^{\omega} \cap M_\alpha[r\circ\pi] \mid \pi\in \Pi \cs$. 
The argument in~\cite{Cohen-random} relies on $\os A_{\alpha}\mid \alpha\in\omega_{1} \cs$
being a pathway. We show that this is not the case. 

Fix $\mathcal{P}=\os  P_{n} \mid n\in\omega \cs \subseteq {[\omega]}^{\omega}$ 
a partition of $\omega$ and let $\mathrm{Q}= \os q_{n}\mid n\in\omega \cs$ 
be an enumeration of the rational numbers.
Furthermore, we take both $\mathcal{P}$ and the enumeration of $\mathrm{Q}$ to
be definable. 
For $f,g \colon \omega\to 2$ we define 
$f\mathbin{\star} g \colon \mathrm{Q}\to 2$ by declaring 
$f\mathbin{\star} g ( q_{n})= 1$ if and only if $f \restriction P_{n} = g\restriction P_{n}$. 
The following proposition implies that no $A_{\alpha}$
is closed under the $\join$ operation.

\begin{proposition}
Let $r \colon \omega_{2}\to 2$ be a $\mathbf{B}(\omega_{2})$-generic 
function over $V$, and let $M$ be a countable elementary submodel of $H(\theta)$. 
There are $\pi_{0},\pi_{1} \in \Pi$ 
such that there is no 
$\sigma \in \Pi$ 
for which $M[r\circ\pi_{0}] \cup M[r\circ\pi_{1}] \subseteq M [r\circ\sigma]$. 
\end{proposition}
\begin{proof}
Let $\delta = M\cap\omega_{1}$. Since $\delta$ is countable ordinal, 
there is $S\subseteq\mathrm{Q}$ order isomorphic to $\delta$. 
Now choose the functions 
$\pi_{0},\pi_{1} \in \Pi$ 
such that the following holds:
\begin{itemize}
\item If $q_{n} \in S$, then $\pi_{0}\restriction P_{n}=\pi_{1}\restriction P_{n}$.
\item If $q_{n}\notin S$, then $\pi_{0}[P_{n}] \cap \pi_{1}[P_{n}] =\emptyset$.
\end{itemize}

Recall that both $M[r\circ \pi_{0}]$ and $M[r\circ\pi_{1}]$ 
are $\mathbf{B}$-generic extensions of $M$. 
Assume that $\os r\circ \pi_{0}, r\circ \pi_{1}\cs \subset M[r\circ\sigma]$ 
for some $\sigma \in\Pi$. 
Then also $(r\circ \pi_{0}) \mathbin{\star} (r\circ \pi_{1}) \in M[r \circ\sigma]$, 
and a simple genericity argument implies 
${\left({(r\circ \pi_{0}) \mathbin{\star} (r\circ \pi_{1})}\right)}^{-1}(1) = S \in M[r \circ\sigma]$. 
Now $\delta \in M[r \circ\sigma]$ is a contradiction with
$M[r \circ\sigma]$ being a generic extension of $M$.
\end{proof}

\pagebreak[0]
\section*{Acknowledgments}

\noindent
The authors would like to thank Jindřich Zapletal
for multiple inspiring conversations
and for suggesting the argument used to prove Theorem~\ref{thm:Silver-product}.
The authors would also like to thank Michael Hrušák and Jonathan Verner for valuable
discussions on the subject.

\bibliography{references}
\bibliographystyle{amsalpha}

\end{document}